\documentclass[11pt,reqno]{amsart}
\usepackage{amsmath}
\usepackage{amsfonts}
\usepackage{dsfont}
\usepackage{mathrsfs}
\usepackage{amssymb}
\usepackage{amsthm}
\usepackage{enumitem}
\usepackage{bbm}
\usepackage{times}
\usepackage{tabularx}
\usepackage{amsbsy}
\usepackage{mathtools}
\usepackage{tikz}
\usetikzlibrary{arrows,decorations.markings}
\usetikzlibrary{matrix}
\usetikzlibrary{graphs}
\usetikzlibrary{backgrounds}
\usepackage{setspace}     \spacing{1}
\usepackage{color}

\usepackage[colorlinks=true,linkcolor=blue,citecolor=blue]{hyperref}
\usepackage{amssymb,latexsym}
\usepackage{mathrsfs,hyperref}
\usepackage{amsmath,latexsym}
\usepackage{amscd,latexsym} 
\usepackage{float}
\usepackage{indentfirst}
\usepackage{amsfonts, latexsym}
 \usepackage{graphicx}

\newtheorem{thm}{Theorem}
\newtheorem{prop}{Proposition}
\newtheorem{lem}[prop]{Lemma}

\theoremstyle{definition}

\theoremstyle{remark}

\newtheorem{rmk}[prop]{Remark} 

\newcommand{\F}{{\mathbb{F}}}
\newcommand{\R}{{\mathbb{R}}}
\newcommand{\Z}{{\mathbb{Z}}}

\newcommand{\Q}{{\mathbb{Q}}}

\newcommand{\M}{\mathcal{M}}

\newcommand{\ham}{\mathcal{H}am}
\newcommand{\mv}{\mu_{MV}}
\newcommand{\fw}{\varphi^t_{\varepsilon {H_f}}}
\newcommand{\fc}{\varepsilon {H_f}}

\newcommand{\cls}{c_{LS}}

\newcommand{\overbar}{\overline}

\newcommand{\al}{\alpha}

\newcommand{\ga}{\gamma}

\newcommand{\cA}{\mathcal{A}}

\newcommand{\cD}{\mathcal{D}}

\newcommand{\cH}{\mathcal{H}}

\newcommand{\cJ}{\mathcal{J}}

\newcommand{\cP}{\mathcal{P}}

\newcommand{\cT}{\mathcal{T}}
\newcommand{\cM}{\mathcal{M}}

\newcommand{\brat}[1]{{\left< #1 \right>}}
\newcommand{\tens}{\otimes}
\newcommand{\ah}{\mathcal{A}_{H}}

\DeclareMathOperator{\im}{\mathrm{Im}}
\DeclareMathOperator{\spec}{\mathrm{Spec}}

\DeclareMathOperator{\crit}{{\mathrm{Crit}}}
\DeclareMathOperator{\ind}{{\mathrm{ind}}}
\DeclareMathOperator{\virdim}{\mathrm{vir-dim}}

\numberwithin{equation}{section}

\title[Cuplength estimates]{A short proof of cuplength estimates on Lagrangian intersections}

\author{Wenmin Gong}

\address{School of Mathematical Sciences, Beijing Normal University, Beijing, 100875, China}

\email{ wmgong@bnu.edu.cn}
\begin{document}
\maketitle

\begin{abstract}
	In this note we give a short proof of  Arnold's conjecture for the zero section of a cotangent bundle of a closed manifold. The proof is based on some basic properties of Lagrangian spectral invariants from Floer theory.

\end{abstract}

\maketitle

\section{Introduction}\label{sec:1}

Let $M$ be a $n$-dimensional closed manifold. We denote by $\omega$ the canonical symplectic structure on the cotangent bundle $T^*M$, which is given by $\omega=-d\theta$ with the Liouville one-form $\theta=pdq$. Given $H\in C^\infty([0,1]\times T^*M)$, the Hamiltonian vector field $X_H$ is determined by $dH=-\omega(X_H,\cdot)$. The flow of $X_H$ is denoted by $\varphi_H^t$ and its time-one map by $\varphi_H:=\varphi_H^1$. Denote by $\ham_c(M,\omega)$  the set of all Hamiltonian diffeomorphisms with compact support. Clearly, for any asymptotically constant Hamiltonian $H$ we have $\varphi_H\in \ham_c(M,\omega)$. Fix a ground field $\F$, eg $\Z_2$,  $\R$, or $\Q$. The singular homology of a topological space $X$ with coefficients in $\F$ is denoted by $H_*(X)$. The \textit{$\F$-cuplength} $cl(M)$ of $M$ is by definition the maximal integer $k$ such that there exist homology classes $u_1,\ldots, u_{k-1}$ in the homology $H_*(M)$ as a ring (with the intersection product)  with $\deg (u_i)<\dim (M) $ such that
$$u_1\cap\cdots\cap u_{k-1}\neq0,$$
where $\cap$ denotes the intersection product.

The goal of this note is to use spectral invariants from Floer theory to reprove the cuplength estimate:

\begin{thm}\label{e:mthm}
	Let $O_M$ denote the zero section of $T^*M$. Then we have
	$$\sharp\big(\varphi(O_M)\cap O_M\big)\geq cl(M),\quad\forall\varphi\in \ham_c(M,\omega).$$
\end{thm}

The above estimate is a special case of the Arnold conjectures~\cite{Ar}.
This case has already been solved by different approaches, for instance,  by Chaperon~\cite{Ch} for the cotangent bundle of torus using variational methods given by Conley and Zehnder~\cite{CZ}, and for general cotangent bundles by Hofer~\cite{Ho1} applying Lyusternik-Shnirelman category theory,  and Laudenbach and Sikorav~\cite{LS} employing a finite dimensional method of ``broken extremals", etc.
All available proofs listed above were given by finite dimensional homological methods essentially.
The method of our proof of Theorem~\ref{e:mthm} is making good use of the properties of Lagrangian spectral invariants~\cite{Oh1,Oh2} from Floer theory which, roughly speaking, is a version of infinite dimensional Morse theory.

\begin{rmk}
	By a modification of the method used here, one can prove a slightly general case of Arnold's conjecures: If $L$ is a closed Lagrangian submanifold of a smooth tame symplectic manifold $(P,\omega)$ satisfying $\pi_2(M,L)=0$, then for any $\varphi\in\ham_c(P,\omega)$, $L\cap\varphi(L)$ has at least $cl(L)$ points. This case was independently proved by Floer~\cite{Fl2} and Hofer~\cite{Ho2}. It seems to the author that the spectral method from Floer theory fits in the degenerate Arnold conjecture very well. For the case that $L$ is a closed monotone Lagrangian submanifold of  of a smooth tame symplectic manifold $P$, we refer to~\cite{Go} for partial results about the Arnold conjecture on Lagrangian intersections in degenerate (non-transversal) sense. Maybe other methods (eg Hofer's) could also be modified to deal with this case, but we are not aware of any reference about it at the time of writing.
	As far as the author knows the degenerate Arnold's conjecture in the monotone case is not understood quite well, and even for a weaker problem, ie  estimating the number of intersections of a closed Lagrangian submanifold with itself under Hamiltonian flows without the nondegenerate assumption. We believe that the spectral method from Floer theory would provide further potential value for attacking these kinds of problems.
	
\end{rmk}

\section{Spectral invariants}\label{sec:2}

\subsection{The minmax critical values}

Let $X$ be a closed $n$-dimensional manifold $X$ and let
$f\in C^\infty(X)$. For any $\mu\in\R$ we put $$X^\mu:=\{x\in X|f(x)<\mu\}.$$
To a non-zero singular homology class $\al\in H_*(X)$, we associate a  numerical invariant by
$$c_{LS}(\al,f)=\inf\{\mu\in\R|\al\in\im(i^\mu_*)\},$$
where $i^\mu_*:H_*(X^\mu)\to H_*(X)$ is the map induced by the natural inclusion $i^\mu:X^\mu\to X$. This number is a critical value of $f$. The function $c_{LS}:H_*(X)\setminus\{0\}\times C^\infty(X)$ is often called a \textit{minmax critical value selector}. The following proposition summarizing the properties of the resulting function, which can be easily extracted from the classical Ljusternik--Schnirelman theory, see, e.g., ~\cite{Cha, HZ,Vi2,CLOT, GG}.

\begin{prop}\label{pp:minmax}
	The minmax critical value selector $c_{LS}$ satisfies the following properties.
	\begin{enumerate}
		\item[{\rm (1)}] $c_{LS}(\al,f)$ is a critical value of $f$, and $c_{LS}(k\al,f)=c_{LS}(\al,f)$ for any nonzero $k\in\F$.
		\item[{\rm(2)}] $c_{LS}(\al,f)$ is Lipschitz in $f$ with respect to the $C^0$-topology.
		\item[{\rm(3)}] Let $[pt]$ and $[X]$ denote the class of a point and the fundamental class respectively. Then
		$$c_{LS}([pt],f)=\min f\leq c_{LS}(\al,f)\leq\max f= c_{LS}([X],f).$$
		\item[{\rm (4)}] $c_{LS}(\al\cap \beta,f)\leq c_{LS}(\al,f)$ for any $\beta\in H_*(X)$ with $\al\cap \beta\neq 0$.
		\item[{\rm(5)}] If $\beta\neq k[X]$ for some $k\in\F$ and $c_{LS}(\al\cap \beta,f)= c_{LS}(\al,f)$, then the set $\Sigma=\{x\in\crit(f)|f(x)= c_{LS}(\al,f)\}$ is homologically non-trivial.
	\end{enumerate}
\end{prop}
Here a subset $S$ of a topological space $X$ is called \textit{homologically non-trivial} in $X$ if for every open neighborhood $U$ of $S$ the map $i_*:H_k(U)\to H_k(X)$ induced by the inclusion $i:U\hookrightarrow X$ is non-trivial.

\subsection{The Lagrangian spectral invariants}
In this subsection we briefly recall the construction of Lagrangian spectral invariants for Hamiltonian diffeomorphism mainly following Oh~\cite{Oh1,Oh2}, see also~\cite{MVZ,Oh3}. Denote $\cH_{ac}$ the set of Hamiltonians $H\in C^\infty([0,1]\times T^*M)$ which are asymptotically constant at infinity.

For $H\in \cH_{ac}$,
the  action functional is defined as
$$\ah(\ga)=\int^1_0H(t,\ga(t))dt-\int\ga^*\theta $$
on the space of paths in $T^*M$
$$\cP=\big\{\ga:[0,1]\to T^*M\big|\ga(0),\ga(1)\in O_M\big\}.$$
Set $L=\varphi_H^1(O_M)$. We define the \textit{Lagrangian action spectrum} of $H$ on $T^*M$ by
$$\spec(L,H)=\big\{\ah(\ga)\big|\ga\in\crit(\ah)\big\}.$$
This is a compact subset of $\R$ of measure zero, see for instance~\cite{Oh1}.

Given a generic $H\in\cH_{ac}$, the intersection $\varphi_H^{1}(O_M)\cap O_M$ is transverse and hence $\crit(\ah)$ is finite. There is an integer-valued index, called the \textit{Maslov-Viterbo index}, $\mv:\crit(\ah)\to \Z$ which is normalized so that if $H:T^*M\to\R$ is a lift of a Morse function $f$ then $\mv$ coincides with the Morse index of $f$.

Denote by $CF_k^{<a}(L,H)$, where $a\in(-\infty,\infty]$ is not in $\spec(L,H)$, the vector space of formal sums
$$\sum_{x_i\in\cP}\sigma_ix_i,$$
where $\sigma_i\in\F$, $\mv(x_i)=k$ and $\ah(x_i)<a$. The graded $\F$-vector space $CF_k^{<a}(L,H)$ has the Floer differential counting the anti-gradient trajectories of the action functional in the standard way whenever a time-dependent almost complex structure compatible with $\omega$ is fixed and the regularity requirements are satisfied, see for instance~\cite{Fl1,Oh1}. As a consequence, we have a filtration of the total Lagrangian Floer complex $CF_*(L,H):=CF_*^{(-\infty,\infty)}(L,H)$. Since the resulting homology, the \textit{filtered Lagrangian Floer homology} of $H$, does not depend on $H\in\cH_{ac}$ (due to continuation isomorphisms), one can extend this construction to all  asymptotically constant Hamiltonians. Let $H\in\cH_{ac}$ be an arbitrary Hamiltonian and let $a$ be outside of $\spec(L,H)$. We define
$$HF^{<a}_*(L,H)=HF^{<a}_*(L,\widetilde{H}),$$
where $\widetilde{H}$ is a $C^2$-small perturbation of $H$ so that
$\varphi_{\widetilde{H}}^{1}(O_M)\cap O_M$ is transverse.
It is not hard to see that $HF^{<a}_*(L,\widetilde{H})$ is independent of $\widetilde{H}$ provided that $\widetilde{H}$ is sufficiently close to $H$.

We denote by $i_*^a:HF^{<a}_*(L,H)\to HF_*(L,H)$ the induced inclusion maps. It is well known that for $H_f=\pi^*f$ where $\pi:T^*M\to M$ is the projection map and $f$ is a Morse function on $M$, $HF_*(L,H_f)$ is canonically isomorphic to the singular homology $H_*(M)$, and hence $H_*(M)\cong HF(L,H)$ for all $H\in\cH_{ac}$. Using this identification, for $\al\in H_*(M)$ and $H\in\cH_{ac}$ we define
$$\ell(\al,H)=\inf\big\{a\in\R\setminus\spec(L,H)|\al\in\im (i^a_*)\big\}.$$
By convention, we have $\ell(0,H)=-\infty$.

\begin{prop}\label{pp:lsi}
	The Lagrangian spectral invariant $\ell:H_*(M)\setminus\{0\}\times\cH_{ac}\to\R$ has the following properties:
	\begin{enumerate}
		\item[{\rm (a)}]  $\ell(\al,H)\in\spec(L,H)$, in particular it is a finite number.
		\item[{\rm (b)}]  $\ell$ is Lipschitz in $H$ in the $C^0$-topology.
		\item[{\rm (c)}] $\ell([M],H)=-\ell([pt],\overbar{H})$ with $\overbar{H}(t,H)=-H(-t,H)$.
		\item[{\rm (d)}]  $\ell([pt],H)\leq \ell(\al,H)\leq \ell([M],H)$ for all $\al\in H_*(M)\setminus\{0\}$.
		\item[{\rm (e)}]  $\ell(\al,H)=\ell(\al,K)$, when $\varphi_H=\varphi_K$ in the universal covering of the group of  Hamiltonian diffeomorphisms, and $H,K$ are normalized.
		\item[{\rm (f)}] $\ell(\al\cap\beta,H\sharp K)\leq \ell(\al,H)+\ell(\beta,K)$, where $(H\sharp K)(t,x)=H(t,x)+K(t,(\varphi_H^t)^{-1}(x))$.
		\item[{\rm (g)}]  If $\varphi_H(O_M)=\varphi_K(O_M)$, then there exists $C\in\R$ such that $\ell(\al,H)=\ell(\al,K)+C$ for all $\al\in H_*(M)\setminus\{0\}$.
		\item[{\rm (h)}] Let $f:M\to\R$ be a smooth function, and let $H_f:T^*M\to\R$ denote a compactly supported autonomous Hamiltonian so that $H_f=f\circ\pi$ on a ball bundle $T^*_RM:=\{(q,p)\in T^*M||p|\leq R\}$ containing $L^f:=\{(q,\partial_q f(q))\in T^*M|q\in M\}$, and $H_f=0$ outside $T^*_{R+1}M$ in $M$, where $|\cdot|$ is the norm induced by a metric $\rho$ on $M$, and $\pi:T^*M\to M$ is the natural projection map. Then $\ell(\al,H_f)=c_{LS}(\al,f)$ for all $\al\in H_*(M)\setminus\{0\}$.
	\end{enumerate}
\end{prop}

\section{The proof of the main theorem}\label{sec:3}

Our main theorem follows immediately from the following lemma.

\begin{lem}\label{lm:strict}
	Let $H\in \cH_{ac}$ and $\alpha,\beta\in H_*(M)\setminus\{0\}$ with $\deg(\al)<n$. If the intersections of $O_M$ and $\varphi_H(O_M)$ are isolated, then
	$$\ell(\al\cap\beta,H)<\ell(\beta,H).$$
	
\end{lem}

\begin{proof}
	Since the intersections of $O_M$ and $\varphi_H(O_M)$ are isolated,  we can pick a small open neighborhood $U$ of $O_M\cap \varphi_H(O_M)$ in $M$ so that $H_k(\overbar{U})=0$ for all $k>0$. Let $f:M\to\R$ be a $C^2$-small function such that $f=0$ on $\overbar{U}$ and $f<0$ on $M\setminus \overbar{U}$, and let $H_f$ be the lift of $f$ to the cotangent bundle $T^*M$ as in Proposition~\ref{pp:lsi}(h). We claim that for any $\al\in H_{<n}(M)$, it holds that
	\begin{equation}\label{e:stricineq}
	\ell(\al,H_f)<0.
	\end{equation}

	For this end, we first prove that  $\cls(\al,f)<0$ for all $\al\in H_{<n}(M)$. In fact, if there exists a homology class $\al_l\in H_{<n}(M)$ such that $\cls(\al_{l},f)=0$, then we have $\cls(\al_l\cap [M],f)=\cls(\al_{l},f)=0$. It follows from Proposition~\ref{pp:minmax}(3) that $\cls([M],f)=\max_Mf=0$.   So we have $\cls(\al_l\cap [M],f)=\cls([M],f)$ with $\al_l\in H_{<n}(M)\setminus\{0\}$. Then by Proposition~\ref{pp:minmax}(5) the zero level set $\overbar{U}$ of $f$ is homologically non-trivial -- a contradiction. Therefore, for any
	$\al\in H_{<n}(M)$, we have $\cls(\al,f)<0$. This, together with Proposition~\ref{pp:lsi}(h), yields $\ell(\al,H_f)<0$.
	
	Next we  show that for sufficiently small $\varepsilon>0$
	\begin{equation}\label{e:eq}
	\ell(\al\cap\beta,H)=\ell(\al\cap\beta,  \varepsilon H_f\sharp H).
	\end{equation}
	Observe that $\varphi^t_{H_f}(q,p)=(q,p+t\partial_q f(a))\in T^*M$ for $t\in[0,1]$ and $(q,p)\in T^*_RM$. Set $L^H_R=\varphi_H^1(O_M)\cap T^*_RM$. Then we have
	$$\varphi_{\varepsilon H_f}(L^H_R)=\{(q,p+\varepsilon df(q))|(q,p)\in L^H_R\}.$$
	Since $L_R^H\cap\pi^{-1}(O_M\setminus U)$ is compact and has no intersections with $O_M$, we deduce that for small enough $\varepsilon>0$, $\varphi_{\varepsilon H_f}(L^H_R)\cap \pi^{-1}(O_M\setminus U)$ has no intersections with $O_M$ as well. For $(q,p)\in T^*M$ with $R\leq\|(q,p)\|\leq R+1$ we have
	$$d_g\big(\varphi_{\varepsilon H_f}(q,p),(q,p)\big)\leq\bigg\|\int^1_0\frac{d}{dt}
	\varphi_{\varepsilon H_f}^t(q,p)dt\bigg\|\leq \varepsilon\sup\limits_{R\leq\|(q,p) \|\leq R+1}\|X_{H_f}\|,$$
	where $d_g$ is the distance function induced by some Riemannian metric $g$ on $M$. Therefore, for  sufficiently small $\varepsilon>0$, $\varphi_{\varepsilon H_f}(T^*_{R+1}M\setminus T^*_RM)$ does not intersect $O_M$. Note that the Hamiltonian diffeomorphism $\varphi_{H_f}^\varepsilon$ is supported in $T^*_{R+1}M$, we conclude that  $\varphi_{\varepsilon {H_f}}\varphi_H(O_M)\cap\pi^{-1}(O_M\setminus U)$ does not intersect $O_M$ provided that $\varepsilon>0$ is sufficiently small. On the other hand, we have that
	$\varphi_{\varepsilon {H_f}}\varphi_H(O_M)\cap\pi^{-1}(U)=\varphi_H(O_M)\cap\pi^{-1}(U)$ because $f=0$ on $U$. So if $\varepsilon>0$ is sufficiently small then the Lagrangians $\varphi_{\varepsilon {H_f}}\varphi_H(O_M)$ and $\varphi_H(O_M)$ have the same intersections with $O_M$. A direct calculation shows that for every such intersection point, the two action values corresponding to $\varepsilon {H_f}\sharp H$ and $H$ are the same. Indeed, there is a one-to-one correspondence between the set $\varphi_H(O_M)\cap O_M$ and the set $\cP(H):=\{x\in\cP|\dot{x}=X_H(x(t))\}$ of Hamiltonian chords by sending $q\in \varphi_H(O_M)\cap O_M$ to $x=\varphi_H^t(\varphi^{-1}_H(q))$. So we get a bijective map defined by
	$$\Upsilon:\crit(\mathcal{A}_H)\longrightarrow\crit(\cA_{\varepsilon {H_f}\sharp H}),\quad x(t)\longmapsto\varphi_{\varepsilon {H_f}}^t(x(t)).$$
	
	Notice that the Hamiltonian flow has the following property
	\[(\fw)^*\theta-\theta=dF_t,\] where the function $F:[0,1]\times T^*M\to\mathbb{R}$ is given by $F_t=\int^t_0(\theta(X_{\fc})-\fc)\circ\varphi_{\fc}^sds$, see for
	instance~\cite[Proposition~9.3.1]{MS}.
	As a consequence, for any $x\in \mathcal{P}(H)$ we have
	\[
	\frac{d}{dt}F_t(x(t))=dF_t(\dot{x}(t))+\big(\theta(X_{\fc})-\fc\big)\circ\varphi_{\fc}^t\big(x(t)\big).
	\]
	which implies
	\[
	(\fw)^*\theta\big(\dot{x}(t)\big)=\theta\big(\dot{x}(t)\big)
	+\frac{d}{dt}F_t(x(t))-\big(\theta(X_{\fc})-\fc\big)\circ\varphi_{\fc}^t\big(x(t)\big).
	\]
	Then we compute
	\begin{eqnarray}
	\cA_{\fc\sharp H}\big(\Upsilon(x(t))\big)&=&\int^1_0\fc(\varphi_{\fc}^t(x(t)))dt+\int^1_0H_t\circ(\varphi_{\fc}^t)^{-1}(\varphi_{\fc}^t(x(t)))dt\notag\\
	&&-\int^1_0\theta\bigg(\frac{d}{dt}\fw\big(x(t)\big)\bigg)dt\notag\\
	&=&\int^1_0\fc(\varphi_{\fc}^t(x(t)))dt+\int^1_0H_t\big(x(t)\big)dt-\int^1_0(\fw)^*\theta\big(\dot{x}(t)\big)dt\notag\\
	&&-\int^1_0\theta\big(X_{\fc}(\fw(x(t)))\big)dt\notag\\
	&=&\mathcal{A}_H\big(x(t)\big)+\int^1_0\fc(\varphi_{\fc}^t(x(t)))-
	\frac{d}{dt}F_t\big(x(t)\big)dt\notag\\
	 &&+\int^1_0\big(\theta(X_{\fc})-\fc\big)\circ\varphi_{\fc}^t\big(x(t)\big)dt-\theta\big(X_{\fc}(\fw(x(t)))\big)dt\notag\\&=&\mathcal{A}_H\big(x(t)\big)+F_1(x(1))-F_0(x(0))\notag\\
	&=&\mathcal{A}_H\big(x(t)\big),
	\end{eqnarray}
	where in the last equality we have used the fact that the value of an autonomous Hamiltonian $H_f$ is constant along its Hamiltonian flow, and $f=0$ on $U$ which contains $x(1)$. Therefore,  the action spectra $\spec(L,\varepsilon H_f\sharp H)$ and $\spec(L,H)$ are the same. Now fix a sufficiently small $\varepsilon>0$ and consider the family of Lagrangians $\varphi_{s\varepsilon {H_f}}\varphi_H(O_L)$ with $s\in[0,1]$. As before, the action spectra $\spec(L,s\varepsilon H_f\sharp H)$, $s\in[0,1]$ are all the same.  Since the action spectrum is a closed nowhere dense subset of $\R$, it follows from Proposition~\ref{pp:lsi}(b)  that
	$\ell(\al\cap\beta, s\varepsilon H_f\sharp H)$ do not depend on $s$. So we have $\ell(\al\cap\beta,H)=\ell(\al\cap\beta,\varepsilon H_f\sharp H)$.
	
	Combining (\ref{e:stricineq}) and (\ref{e:eq}),  it follows from Proposition~\ref{pp:lsi}(f) that
	$$\ell(\al\cap\beta,H)=\ell(\al\cap\beta,\varepsilon H_f\sharp H)\leq \ell(\al, \varepsilon H_f)+\ell(\beta,H)<\ell(\beta, H).$$
	This completes the proof of the lemma.
\end{proof}

\begin{proof}[The proof of Theorem~\ref{e:mthm}] 	Without loss of generality we may assume that
	the intersections of $O_M$ and $\varphi_H(O_M)$ are isolated, otherwise, nothing needs to be proved.
	Set $cl(M)=k+1$. By definition there exist $u_i\in H_{<n}(M)$, $i=1\ldots,k$ such that $u_1\cap\cdots\cap u_k=[pt]$. We put
	$$[M]=\al_0,\al_1,\ldots,\al_k\in H_*(M), \quad \al_i=u_{k-i+1}\cap\al_{i-1}.$$
	For any $\varphi\in \ham_c(M,\omega)$, there exists a Hamiltonian $H\in\cH_{ac}$ such that $\varphi=\varphi_H^1$.
	It follows from  Lemma~\ref{lm:strict} and Proposition~\ref{pp:lsi}(a) that there exist $k+1$ elements $x_i\in\crit(\ah)$  such that
	$$\ell(\al_k,H)=\ah(x_k)<\ell(\al_{k-1},H)=\ah(x_{k-1})<\cdots<\ell(\al_0,H)=\ah(x_0).$$
	Hence, all $x_i$, $i=0,\ldots, k$ are different. The one-to-one correspondence  between the intersection points of $O_M$ and $\varphi_H^1(O_M)$ and the critical points of $\ah$ concludes the desired result.

\end{proof}


\begin{thebibliography}{SK}
	
     \bibitem{Ar}  Arnold, V. I.,  Sur une propri\'{e}t\'{e} topologique des applications canoniques de la m\'{e}canique classique, {\it C. R. Acad. Sci. Paris}  {\bf 261} (1965) 3719--3722.
	 	
	 	
	 \bibitem{Ch} Chaperon, M.,  Quelques questions de g\'{e}om\'{e}trie symplectique, S\'{e}minaire Bourbaki, {\it  Asterisque}, {\bf 105--106} ( 1982-1983), 231--249.
	
	 \bibitem{Cha}  Chang, K.-C., \textit{ Infinite Dimensional Morse Theory and Multiple Solution Problem}, Birkha\"{a}user, 1993.
	
    \bibitem{CZ} Conley, C. and Zehnder, E., The Birkhoff-Lewis fixed point theorem and a conjecture of V. I. Arnold,  {\it  Invent. Math.} {\bf 73} (1983), 33--49.
	
	 \bibitem{CLOT} Cornea, O., Lupton, G., Oprea, J. and Tanr\'{e} D., \textit{Lusternik-Schnirelmann category}, volume {\bf 103} of \textit{Mathematical Surveys and Monographs}. American Mathematical Society, Providence, RI, 2003.
	
	\bibitem{Fl1} Floer, A,. Morse theory for Lagrangian intersections. {\it J. Differential Geom.} {\bf 28} (1988),   513--547.
	
	 \bibitem{Fl2} Floer, A., Cuplength estimates on Lagrangian intersections. {\it Comm. Pure Appl. Math.}  {\bf 42} (1989), 335--356.
	
	   \bibitem{GG} Ginzburg, V. L. and G\"{u}rel, B. Z., Action and index spectra and periodic orbits in Hamiltonian dynamics. {\it  Geom. Topol.}  {\bf 13} (2009),  2745--2805.
	
	  \bibitem{Go} Gong, W., Lagrangian intersections and a conjecture of Arnol'd. arXiv:2111.15442v5, Nov. (2021). To appear in Israel Journal of Mathematics. 
	
	\bibitem{Ho1}  Hofer, H.,  Lagrangian Embeddings and Critical Point Theory, {\it Ann. Inst. H. Poincar\'{e} Anal. Non Lin\'{e}aire} {\bf 2} (1985), 407--462.
	
	\bibitem{Ho2} Hofer, H., Lusternik-Schnirelman-theory for Lagrangian intersections. {\it Ann. Inst. H. Poincar\'{e} Anal. Non Lin\'{e}aire} {\bf 5} (1988),  465--499.
	
	
	\bibitem{HZ} Hofer, H. and Zehnder, E., \textit{Symplectic invariants and Hamiltonian dynamics}. Birkhäuser Advanced Texts: Basler Lehrbücher. Birkh\"{a}user Verlag, Basel, 1994. xiv+341 pp.
	
	
	 \bibitem{LS}Laudenbach, F. and Sikorav, J.-C., Persistance d'intersection avec la section nulle au cours d'une isotopie hamiltonienne dans un fibr\'{e} cotangent.  {\it Invent. Math.} {\bf 82} (1985),  349--357.


       \bibitem{MS} McDuff, D. and Salamon, D., \textit{ Introduction to symplectic topology. Second edition.} Oxford Mathematical Monographs. The Clarendon Press, Oxford University Press, New York, 1998.

    \bibitem{MVZ} Monzner, A., Vichery, N. and  Zapolsky, F., Partial quasi-morphisms and quasi-states on cotangent bundles, and symplectic homogenization.  {\it J. Mod. Dyn. } {\bf 6} (2012),  205--249.

     \bibitem{Oh1} Oh, Y.-G., Symplectic topology as the geometry of action functional. I. Relative Floer theory on the cotangent bundle. {\it J. Differential Geom.} {\bf 46} (1997), 499--577.

     \bibitem{Oh2} Oh, Y.-G.,  Symplectic topology as the geometry of action functional. II. Pants product and cohomological invariants. {\it Comm. Anal. Geom.} {\bf 7} (1999), 1--54.

       \bibitem{Oh3} Oh, Y.-G.,  Geometry of generating functions and Lagrangian spectral invariants. arXiv:1206.4788, June (2012).

	
	
	\bibitem{Vi2} Viterbo, C.,  Some remarks on Massey products, tied cohomology classes, and the Lusternik-Shnirelman category. {\it Duke Math. J.} {\bf 86 } (1997), 547--564.
	
\end{thebibliography}
 \end{document}